\UseRawInputEncoding
\documentclass[12pt, reqno]{amsart}
\usepackage[margin=1in]{geometry}
\usepackage{amssymb,latexsym,amsmath,amscd,amsfonts}
\usepackage{latexsym}
\usepackage[mathscr]{eucal}
\usepackage{bm}
\usepackage{graphicx}
\usepackage{forest}
\usepackage{tikz-qtree}
\usepackage{mathptmx}
\usepackage{amssymb}
\usepackage{upgreek}
\usepackage{amsthm}
\usepackage{dcolumn}
\usepackage{tikz-cd}
\usepackage[all]{xy}
\usepackage{enumitem}
\usepackage[utf8]{inputenc}

\def \qed {\hfill \vrule height6pt width 6pt depth 0pt}
\def\textmatrix#1&#2\\#3&#4\\{\bigl({#1 \atop #3}\ {#2 \atop #4}\bigr)}
\def\dispmatrix#1&#2\\#3&#4\\{\left({#1 \atop #3}\ {#2 \atop #4}\right)}
\newcommand{\beg}{\begin{equation}}
	\newcommand{\eeg}{\end{equation}}
\newcommand{\ben}{\begin{eqnarray*}}
	\newcommand{\een}{\end{eqnarray*}}

\newtheorem{thm}{Theorem}[section]
\newtheorem{cor}[thm]{Corollary}

\newtheorem{prop}[thm]{Proposition}
\numberwithin{equation}{section} \theoremstyle{definition}
\newtheorem{defn}[thm]{Definition}
\newtheorem{rem}[thm]{Remark}

\newtheorem{eg}[thm]{Example}

\newcommand{\C}{\mathbb{C}}

\newcommand{\T}{\mathbb{T}}
\newcommand{\N}{\mathbb{N}}

\newcommand{\HS}{\mathcal{H}}

\def\textmatrix#1&#2\\#3&#4\\{\bigl({#1 \atop #3}\ {#2 \atop #4}\bigr)}
\def\dispmatrix#1&#2\\#3&#4\\{\left({#1 \atop #3}\ {#2 \atop #4}\right)}

\title[$q$-commuting $2 \times 2$ scalar matrices]{Classification and dilation for $q$-commuting $2 \times 2$ scalar matrices}

\author[PAL, SAHASRABUDDHE AND TOMAR]{SOURAV PAL, PRAJAKTA SAHASRABUDDHE AND NITIN TOMAR}

\address[Sourav Pal]{Mathematics Department, Indian Institute of Technology Bombay,
	Powai, Mumbai - 400076, India.} \email{souravpal@iitb.ac.in}

\address[Prajakta Sahasrabuddhe]{Mathematics Department, Indian Institute of Technology Bombay, Powai, Mumbai-400076, India.} \email{prajakta@math.iitb.ac.in}	

\address[Nitin Tomar]{Mathematics Department, Indian Institute of Technology Bombay, Powai, Mumbai-400076, India.} \email{tnitin@math.iitb.ac.in}



\begin{document}
	
	\maketitle

	\begin{abstract}
		A tuple $\underline{T}=(T_1, \dotsc, T_k)$ of operators on a Hilbert space $\HS$ is said to be \textit{$q$-commuting with} $\|q\|=1$ or simply $q$-\textit{commuting} if there is a family of scalars $q=\{q_{ij} \in \C : |q_{ij}|=1, \ q_{ij}=q_{ji}^{-1}, \  1 \leq i < j \leq k \}$ such that $T_i T_j =q_{ij}T_j T_i$  for $1 \leq i < j \leq k$. Moreover, if each $q_{ij}=-1$, then $\underline{T}$ is called an \textit{anti-commuting tuple}. A well-known result due to Holbrook \cite{Holbrook} states that a commuting $k$-tuple consisting of $2 \times 2$ scalar matrix contractions always dilates to a commuting $k$-tuple of unitaries for any $k\geq 1$. To find a generalization of this result for a $q$-commuting $k$-tuple of $2\times 2$ scalar matrix contractions, we first classify such tuples into three types upto similarity. Then we prove that a $q$-commuting tuple which is unitarily equivalent to any of these three types, admits a $\widetilde{q}$-unitary dilation, where $\widetilde q \subseteq q \cup \{1\}$. A special emphasis is given to the dilation of an anti-commuting tuple of $2 \times 2$ scalar matrix contractions.  
	\end{abstract}

\section{Introduction}\label{sec01}
	
		\vspace{0.2cm}
	
	\noindent Throughout the paper, all operators are bounded linear maps acting on complex Hilbert spaces. We denote by $\mathbb{C}, \mathbb{D}$ and $\mathbb{T}$ the complex plane, the unit disk and the unit circle in the complex plane, respectively, with center at the origin. $\mathcal{B}(\HS)$ denotes the algebra of operators acting on a Hilbert space $\HS$. For an operator $T$, we denote its spectrum by $\sigma(T)$. A \textit{contraction} is an operator with norm not greater than $1$. A scalar multiple of the identity operator is called a \textit{scalar map}. 
	
	\smallskip
	
	In this article, we shall study a $q$-commuting tuple of operators which is defined below.
	\begin{defn}
		A tuple $\underline{T}=(T_1, \dotsc, T_k)$ of operators acting on a Hilbert space $\HS$ is said be \textit{$q$-commuting with} $\|q\|=1$ or simply $q$-\textit{commuting} if there exists a family of scalars $q=\{q_{ij} \in \T: \ q_{ij}=q_{ji}^{-1}, \  1 \leq i < j \leq k \}$ such that $T_i T_j =q_{ij}T_j T_i$ for $1 \leq i < j \leq k$. Additionally, if $T_i T_j^* =\overline{q}_{ij}T_j^*T_i$ for $1 \leq i < j \leq k$, then $\underline{T}$ is said to be \textit{doubly $q$-commuting}.  We say that $\underline{T}$ is an \textit{anti-commuting} tuple if $T_iT_j=-T_jT_i$ for $ 1\leq i<j \leq k$.
\end{defn}
Since we consider only $q$-commuting tuples of contractions with $\|q\|=1$, throughout the paper such a tuple will be simply referred to as a \textit{$q$-commuting tuple} or \textit{$q$-commuting tuple of contractions} without mentioning the condition `$\|q\|=1$'.

\smallskip

Let $\underline{T}=(T_1, \dotsc, T_k)$ be a commuting or $q$-commuting tuple of contractions acting on a Hilbert space $\HS$. A commuting or $\widetilde{q}$-commuting tuple of unitaries $\underline{U}=(U_1, \dotsc, U_k)$ acting on a Hilbert space $\mathcal{K}$ is said to dilate $\underline{T}$ if there is an isometry $V: \HS \to \mathcal{K}$ such that 
		\[
		T_1^{m_1}\dotsc T_k^{m_k}=V^*U_1^{m_1}\dotsc U_k^{m_k}V
		\] 	
		for all non-negative integers $m_1, \dotsc, m_k$. More specifically, if $\underline{U}$ satisfies the same $q$-intertwining relations as $\underline{T}$, i.e. if $q=\widetilde{q}$, then we say that $\underline{U}$ is a $q$-unitary dilation of $\underline{T}$. In 1953, Bela Sz.-Nagy \cite{Nagy 1} proved that every contraction dilates to a unitary. After a few years, Ando \cite{Ando} generalized Sz.-Nagy's result for a pair of commuting contractions, that is, every pair of commuting contractions can be dilated to a pair of commuting unitaries. However, Parrott \cite{Par} showed by an example that a commuting triple of contractions may not admit dilation to a commuting triple of unitaries. Evidently, Parrott's result closes the possibility of having unitary dilation in general for a tuple of commuting contractions $(T_1, \dots, T_k)$ if $k>2$. For further reading on this matter, an interested reader is referred to CH-1 of classic \cite{Nagy}. After Parrott's example, the problem of characterizing the commuting tuples of contractions $(T_1, \dots, T_k)$ remains unresolved till date. In \cite{Holbrook}, Holbrook profoundly found success of unitary dilation for any commuting tuple of contractions $(T_1, \dots , T_k)$ acting on $\C^2$. The result is stated below.
		
\begin{prop}[\cite{Holbrook}, Proposition 3]\label{prop01}
Any $k$-tuple of commuting contractions on $\C^2$ has a unitary dilation.
\end{prop}

The aim of this paper is to investigate if Holbrook's result can be generalized for a $q$-commuting tuple of contractions acting on $\C^2$. Note that in \cite{PalII}, the authors of this article found success of $q$-unitary dilation for several classes of $q$-commuting contractions as mentioned below.

\begin{thm}[\cite{PalII}, Theorem 5.2]\label{thm_dcq}
Let $\mathcal{T}=\{T_\alpha : \alpha \in \Lambda\}$ be a $q$-commuting family of contractions on a Hilbert space $\HS$. Then $\mathcal{T}$ admits a $q$-unitary dilation in each of the following cases.
		\begin{enumerate}
			\item $\mathcal{T}$ consists of isometries.
			\item $\mathcal{T}$ consists of doubly $q$-commuting contractions.
			\item $\mathcal{T}$ is a countable family and ${\displaystyle \sum_{\alpha \in \Lambda} \|T_{\alpha}h\|^2 \leq \|h\|^2 }$ for all $h\in \HS$. 
		\end{enumerate}
	\end{thm}
 In order to find generalization of Proposition \ref{prop01} for a $q$-commuting tuple of $2 \times 2$ scalar matrix-contractions, we first classify such tuples into three categories upto similarity in Section \ref{sec_02}. Let us mention that two operator tuples $\underline{A}=(A_1, \dotsc, A_k)$ and $\underline{B}=(B_1, \dotsc, B_k)$ acting on a Hilbert space $\HS$ are said to be \textit{similar} if there is an invertible operator $P$ on $\HS$ such that $P^{-1}A_jP=B_j$ for $1 \leq j \leq k$. After completely classifying all $q$-commuting tuples of contractions on $\C^2$ into three types upto similarity, we prove in Section \ref{sec_03} that every $q$-commuting tuple of contractions that is unitarily equivalent to any of these three types, admits a $\widetilde{q}$-unitary dilation, where $\widetilde q \subseteq q \cup \{ 1 \}$. This is done in Theorem \ref{thm_dil_III} which is a main result of this article.
 
 \medskip
 
 Interestingly, the dilation issue is more subtle for an anti-commuting tuple of contractions on $\C^2$. In Theorem \ref{thm_ni}, we first prove that if atleast one among $T_1, \dotsc, T_k$ is a non-zero non-invertible contraction, then $\underline{T}$ admits a dilation to $q$-commuting $k$-tuple of unitaries. Consequently, we are only left to consider the case when $T_1, \dotsc, T_k$ are all invertible contractions. In this case, we have a bound on $k$ due to the following result. 

\begin{thm}[\cite{Hrubes}, Theorem 1]\label{thm_Hrubes}
Let $(T_1, \dotsc, T_k)$ be an anti-commuting tuple of $n \times n$ invertible matrices. Then $k \leq 2\log_2 n+1$. The bound is achieved if $n$ is a power of two. 
\end{thm}

It follows from the above theorem that an anti-commuting tuple $(T_1, \dotsc, T_k)$ of invertible contractions on $\C^2$ can have atmost $3$ contractions, i.e. $k\leq 3$. If $k = 2$, then the desired $q$-commuting unitary dilation was obtained by Keshari and Mallick in \cite{K.M.} which is stated below. Also, some of the results of \cite{K.M.} were further generalized in \cite{BPS-1}.

\begin{thm}[\cite{K.M.}, Theorem 2.3]\label{thm_KM}
Let $(T_1, T_2)$ be a $q$-commuting pair of contractions on a Hilbert space $\HS$. Then there exist a $q$-commuting pair of unitaries $(U_1, U_2)$ on a Hilbert space $\mathcal{K}$ and an isometry $V : \HS \to \mathcal{K}$ such that $T_1^mT_2^n=V^* U_1^mU_2^nV$ for all $m, n \in \N \cup \{0\}$.
\end{thm}

Thus, it suffices to consider only anti-commuting triples of invertible contractions. In Theorem \ref{thm_antiII}, which is another main result of this article, we prove that for an anti-commuting triple of invertible contractions $(T_1,T_2,T_3)$ on $\C^2$, there exist an anti-commuting triple of unitaries $(U_1, U_2, U_3)$ and a non-zero scalar $\beta$ such that $(\beta U_1, U_2, U_3)$ dilates $(T_1, T_2, T_3)$. Moreover, such a triple $(T_1,T_2,T_3)$ dilates to a triple of anti-commuting unitaries if $\|T_1\| \leq \|T_2T_3\|$.

\medskip	

\section{The $ 2\times 2$ anti-commuting contractive scalar matrices}\label{sec_04}

\medskip

\noindent Let $T_1, \dots , T_k$ be $q$-commuting contractions on $\C^2$ and let $T_{i}T_j \neq 0$ for atleast one pair $(T_i, T_j)$ with $i \ne j$. If every $q_{ij}$ in the family $q$ is same, i.e. there is some $q_0 \in \T$ such that $T_iT_j=q_0T_jT_i$ for $1 \leq i< j \leq k$, then the only non-trivial possibilities are $q_0=1$ or $q_0=-1$. The $q_0=1$ case, i.e. the commutative case, has been studied by Holbrook in \cite{Holbrook}. In this Section, we study the $q_0=-1$ case, i.e. the anti-commuting case and obtain some dilation results for such a tuple. In this direction, Sebesty\'{e}n \cite{Seb} proved that an anti-commuting pair of contractions on a Hilbert space admits a dilation to an anti-commuting pair of unitaries. Keshari and Mallick \cite{K.M.} generalized this result to any $q$-commuting pair of contractions as mentioned in Theorem \ref{thm_KM}. We first explain below the reason behind considering only non-zero contractions in a $q$-commuting tuple when searching for dilation.
	
\begin{rem}\label{rem_1}
 Let $\underline{T}=(T_1, \dotsc, T_k)$ be a $q$-commuting tuple of contractions acting on a Hilbert space $\mathcal{H}$ and let $\underline{T}$ consist of some zero operators. Without loss of generality let us write $\underline{T}=(T_1, \dotsc, T_m,T_{m+1}, \dotsc, T_k)$, where $T_1, \dotsc, T_m$ are non-zero for some $m \in \{1, \dotsc, k\}$ and $T_{m+1}=\dotsc=T_k=0$. Evidently, $\underline{T}'=(T_1, \dotsc, T_m)$ is a $q'$-commuting tuple of non-zero contractions where $q'=\{q_{ij} : 1 \leq i<j \leq m \}$. We show that it suffices to admit a $q'$-unitary dilation of $\underline{T}'$ to obtain a $q$-unitary dilation of $\underline{T}$. Indeed, if $\underline{U}'=(U_1, \dotsc, U_m)$ acting on a space $\mathcal{K}'$ is a $q'$-unitary dilation of $\underline{T}'$, then $\underline{U}=(U_1, \dotsc, U_m, U_{m+1}, \dotsc, U_k)$ with $U_{m+1}=\dotsc =U_k=0$ is a doubly $q$-commuting tuple of normal contractions on $\mathcal{K}$ that dilates $\underline{T}$. One can now apply Theorem \ref{thm_dcq} on $\underline{U}$ to obtain the desired $q$-unitary dilation for $\underline{T}$. 
\end{rem}
		
From now onwards, we assume that any $q$-commuting tuple $(T_1, \dotsc, T_k)$ consists of non-zero operators. We would like to reach the general case step by step beginning with the following theorem.

\begin{thm}\label{lem_nil}
An anti-commuting $k$-tuple of contractions on $\C^2$ consisting of atleast one non-zero nilpotent matrix admits dilation to a $q$-commuting $k$-tuple of unitaries with $q\subseteq \{\pm 1\}$.
\end{thm}	

\begin{proof}
Let $\underline{T}=(T_1, \dotsc, T_k)$ be an anti-commuting tuple of contractions acting on $\C^2$. Without loss of generality (as explained in Remark \ref{rem_1}), let us assume that each $T_j$ is a non-zero contraction and that $T_1$ is a nilpotent matrix. Since $T_1$ is nilpotent, $T_1$ has the following upper triangular form with respect to an appropriate orthonormal basis of $\C^2$ and also let us assume that the forms of $T_2, \dots, T_k$ with respect to the same basis be
	\[
	T_1=\begin{pmatrix}
		0 & d_1\\
		0 & 0\\
	\end{pmatrix} \quad \mbox{and} \quad 
	T_j=\begin{pmatrix}
		c_j & d_j\\
		e_j & f_j\\
	\end{pmatrix} \quad (2\leq j \leq k).
	\]
For $1<j \leq k$, it is easy to see that $T_1T_j=-T_jT_1$ if and only if $e_j=0$ and $c_j=-f_j$. Then
\[
	T_1=\begin{pmatrix}
		0 & d_1\\
		0 & 0\\
	\end{pmatrix} \quad \mbox{and} \quad 
	T_j=\begin{pmatrix}
		c_j & d_j\\
		0 & -c_j\\
	\end{pmatrix} \quad (1<j \leq k).
	\]
If all $c_j=0$ for $1< j \leq k$, then $T_iT_j=0=T_jT_i$ for $1 \leq i, j \leq k$ and the desired conclusion follows from Proposition \ref{prop01}. Assume that there is some $m \in \{2, \dotsc, k\}$ such that $c_m \ne 0$.	Note that
\[
T_jT_m+T_mT_j=\begin{pmatrix}
		c_j & d_j\\
		0 & -c_j\\
	\end{pmatrix}\begin{pmatrix}
		c_m & d_m\\
		0 & -c_m\\
	\end{pmatrix}+\begin{pmatrix}
		c_m & d_m\\
		0 & -c_m\\
	\end{pmatrix}\begin{pmatrix}
		c_j & d_j\\
		0 & -c_j\\
	\end{pmatrix}=\begin{pmatrix}
		2c_jc_m & 0\\
		0 & 2c_jc_m\\
	\end{pmatrix}
\]
for $1 < j \leq k$ and $j \ne m$. Hence, $T_jT_m=-T_mT_j$ if and only if $c_j=0$ for $1 < j \leq k$ and $j \ne m$. Putting everything together, we have that
\[
T_j=\begin{pmatrix}
		0 & d_j\\
		0 & 0\\
	\end{pmatrix} \quad \mbox{and} \quad 
	T_m=\begin{pmatrix}
		c_m & d_m\\
		0 & -c_m\\
	\end{pmatrix} \quad (1 \leq j \leq k \ \& \ j \ne m).
	\]
Let $n \in \{1, \dotsc, k\}\setminus \{m\}$ be such that $|d_n|=\max\{|d_j| : 1 \leq j \leq k \ \& \ j \ne m\}$. Since each $T_j$ is non-zero, we have that $d_n \neq 0$. For $j \in \{1, \dotsc, k\}\setminus \{m\}$, we define $w_j=d_jd_n^{-1}$ and have
\begin{equation}\label{eqn_318}
T_j=w_jT_n.
\end{equation}
It follows from Theorem \ref{thm_KM} that there exist an anti-commuting pair of unitaries $(U_n, U_m)$ on a Hilbert space $\mathcal{K}$ and an isometry $V: \C^2 \to \mathcal{K}$ such that 
\begin{equation}\label{eqn_319}
T_n^sT_m^t=V^* U_n^sU_m^tV
\end{equation}
for all integers $s, t \geq 0$. Define an operator tuple $\underline{N}=(N_1, \dotsc, N_k)$ as follows:		
\[
		N_j=\left\{
		\begin{array}{ll}
				w_jU_n & \mbox{if} \ 1 \leq j \leq k \ \& \ j \ne m, \\ 
			U_m & \mbox{if} \ j=m. 
		\end{array} 
		\right.
		\]
It is evident that $\underline{N}$ is a $q$-commuting tuple of normal contractions on $\mathcal{K}$. Indeed, we have that
$N_iN_j=q_{ij}N_jN_i$ and $N_iN_j^*=\overline{q}_{ij}N_j^*N_i$ with $q_{ij}=\pm 1$ for $1 \leq i<j \leq k$ since $(U_n, U_m)$ is an anti-commuting pair of unitaries. For non-negative integers $s_1, \dotsc, s_k$, we have by  (\ref{eqn_318}) \& (\ref{eqn_319}) that 
\begin{equation*}
\begin{split}
V^*\left(N_m^{s_m}\prod_{ j \ne m}N_j^{s_j}\right)V
=V^*\left(U_m^{s_m}\prod_{ j \ne m}w_j^{s_j}U_n^{s_j} \right)V
=T_m^{s_m}\prod_{ j \ne m}w_j^{s_j}T_n^{s_j} 
=T_m^{s_m}\prod_{ j \ne m}T_j^{s_j}
\end{split}
\end{equation*}
and so, $\underline{N}$ dilates $\underline{T}$. The desired conclusion now follows from Theorem \ref{thm_dcq}.
\end{proof}

\begin{thm}\label{thm_ni}
An anti-commuting $k$-tuple of contractions on $\C^2$ consisting of atleast one non-zero non-invertible contraction admits a dilation to $q$-commuting $k$-tuple of unitaries with $q \subseteq \{\pm 1\}$.
\end{thm}

\begin{proof}
Let $\underline{T}=(T_1, \dotsc, T_k)$ be an anti-commuting tuple of contractions on $\C^2$. Without loss of generality, we assume that $T_1$ is a non-zero non-invertible contraction. Then, with respect to an appropriate orthonormal basis of $\C^2$, we have that
	\[
	T_1=\begin{pmatrix}
		c_1 & d_1\\
		0 & 0\\
	\end{pmatrix} \quad \mbox{and} \quad 
	T_j=\begin{pmatrix}
		c_j & d_j\\
		e_j & f_j\\
	\end{pmatrix} \quad (1< j \leq k).
	\]
If $c_1=0$, then the desired conclusion follows from Theorem \ref{lem_nil}. We assume that $c_1 \ne 0$. There are two possible cases depending on the invertibility of $T_2, \dotsc, T_k$. Suppose there is some $j \in \{2, \dotsc, k\}$ such that $T_j$ is invertible. Then $T_1=-T_jT_1T_j^{-1}$ and so, $tr(T_1)=0$	which is a contradiction, because $tr(T_1)=c_1 \ne 0$. Consequently, $T_2, \dotsc, T_k$ are non-invertible. For $2 \leq j \leq k$, we have that 
\[
T_jT_1+T_1T_j=\begin{pmatrix}
		c_1c_j & d_1c_j\\
		c_1e_j & d_1e_j\\
	\end{pmatrix}+\begin{pmatrix}
		c_1c_j+d_1e_j & c_1d_j+d_1f_j\\
		0 & 0\\
	\end{pmatrix}=\begin{pmatrix}
		2c_1c_j+d_1e_j & d_1c_j+c_1d_j+d_1f_j\\
		c_1e_j & d_1e_j\\
	\end{pmatrix}.
\]
Hence, $T_1T_j=-T_jT_1$ holds if and only if $e_j=c_j=0$ and $d_j=-d_1f_jc_1^{-1}$ for $ 2 \leq j \leq k$. Then
\[
	T_1=\begin{pmatrix}
		c_1 & d_1\\
		0 & 0\\
	\end{pmatrix} \quad \mbox{and} \quad 
	T_j=f_j \ A, \quad \text{where} \quad A=\begin{pmatrix}
		0 & -d_1c_1^{-1}\\
		0 & 1\\
	\end{pmatrix} \quad (2 \leq j \leq k).
	\]
To avoid the trivial case, we assume that there exists $j \in \{2, \dotsc, k\}$ such that $f_j \ne 0$. Choose $m \in \{2, \dotsc, k\}$ such that $|f_m|=\max\{|f_j| : 2 \leq j \leq k\}$. Then 
\begin{equation}\label{eqn_320}
T_j=w_jT_m \quad \text{where} \quad w_j=f_jf_m^{-1} \quad (2 \leq j \leq k).
\end{equation} 
Applying Theorem \ref{thm_KM} on the anti-commuting pair $(T_1, T_m)$, we have that there is an anti-commuting pair of unitaries $(U_1, U_m)$ on a Hilbert space $\mathcal{K}$ and an isometry $V: \C^2 \to \mathcal{K}$ such that 
\begin{equation}\label{eqn_321}
T_1^sT_m^t=V^* U_1^sU_m^tV
\end{equation}
for all integers $s, t \geq 0$. Define an operator tuple $\underline{N}=(N_1, \dotsc, N_k)$ as follows:		
\[
		N_j=\left\{
		\begin{array}{ll}
				U_1 & \mbox{if} \ j =1, \\ 
			w_jU_m & \mbox{if} \ 2 \leq j \leq k. 
		\end{array} 
		\right.
		\]
It is evident that $\underline{N}$ is a $q$-commuting tuple of normal contractions on $\mathcal{K}$. Indeed, we have that
$N_iN_j=q_{ij}N_jN_i$ and $N_iN_j^*=\overline{q}_{ij}N_j^*N_i$ with $q_{ij}=\pm 1$ for $1 \leq i<j \leq k$ since $(U_1, U_m)$ is an anti-commuting pair of unitaries. For non-negative integers $s_1, \dotsc, s_k$, we have by  (\ref{eqn_320}) \& (\ref{eqn_321}) that 
\begin{equation*}
\begin{split}
V^*\left(N_1^{s_1}N_2^{s_2} \dotsc N_k^{s_k}\right)V
=w_2^{s_2}\dots w_k^{s_k}V^*\left(U_1^{s_1}U_m^{s_2+\dotsc +s_k}\right)V
=w_2^{s_2}\dots w_k^{s_k} T_1^{s_1}T_m^{s_2+\dotsc +s_k}
=T_1^{s_1}T_2^{s_2} \dotsc T_k^{s_k}
\end{split}
\end{equation*}
and so, $\underline{N}$ dilates $\underline{T}$. The desired conclusion now follows from Theorem \ref{thm_dcq}.
\end{proof}

The only remaining case that is to be discussed here (for the dilation of an anti-commuting tuple of contractions on $\C^2$) is when each contraction in the tuple is invertible. It follows from Theorem \ref{thm_Hrubes} that if a $k$-tuple of anti-commuting contractions $(T_1, \dots , T_k)$ on $\C^2$ consists of only invertible operators, then $k \leq 3$. Below we write it as a proposition.
\begin{prop} \label{prop:Sec-2-04}
If $(T_1, \dots , T_k)$ on $\C^2$ is a $k$-tuple of anti-commuting invertible matrices, then $k\leq 3$.
\end{prop}
So, Proposition \ref{prop:Sec-2-04} tells us that it suffices to consider only anti-commuting triples consisting of invertible contractions on $\C^2$. Our first step in this direction is the following proposition.

\begin{prop}\label{prop_anti}
An anti-commuting triple of invertible contractions on $\C^2$ consisting of atleast one normal contraction admits a dilation to an anti-commuting triple of unitaries.
\end{prop}

\begin{proof}
Let $(T_1, T_2, T_3)$ be an anti-commuting triple of invertible contractions on $\C^2$. Without loss of generality, we assume that $T_1$ is normal. Evidently, if $A$ and $B$ are anti-commuting matrices such that $A$ is invertible, then $tr(B)=-tr(A^{-1}BA)=-tr(B)$ and so, $tr(B)=0$. Thus, $tr(T_j)=0$ for $1 \leq j \leq 3$. One can choose an orthonormal basis of $\C^2$ such that 
	\[
	T_1=\begin{pmatrix}
		a_1 & 0 \\
		0 & -a_1\\
	\end{pmatrix} \quad \mbox{and} \quad 
	T_j=\begin{pmatrix}
		a_j & d_j\\
		c_j & -a_j\\
	\end{pmatrix} \quad (j=2,3).
	\]
Since $a_1 \ne 0$, we have that 
\[
0=T_1T_j+T_jT_1=\begin{pmatrix}
2a_1a_j & 0\\
0 & 2a_1a_j\\
\end{pmatrix} \quad \text{and so,} \quad a_j=0 \ \text{for $j=2,3$}.
\]	
Then, it follows that 
\[
T_j=\begin{pmatrix}
0 & d_j \\
c_j & 0 \\
\end{pmatrix} \ (j=2,3) \quad \text{and so,} \quad T_2T_3+T_3T_2=\begin{pmatrix}
d_2c_3+c_2d_3 & 0 \\
0 & d_2c_3+c_2d_3\\
\end{pmatrix}.
\]
By invertibililty of $T_2, T_3$, we have that $c_2, c_3, d_2, d_3$ are all non-zero and so, $d_3=-c_2^{-1}c_3d_2$. Then 
\begin{equation}\label{eqn_322}
T_1=a_1R_{-1}, \quad T_2=c_2c_3^{-1}T_3R_{-1} \quad \text{and} \quad T_3=c_3c_2^{-1}T_2R_{-1}, \quad \text{where} \quad R_{-1}=\begin{pmatrix}
1 & 0 \\
0 & -1
\end{pmatrix}.
\end{equation}
Either $|c_3| \leq |c_2|$ or $|c_2| \leq |c_3|$. Without loss of generality, we assume that $|c_2| \leq |c_3|$. Applying Theorem \ref{thm_KM} on the anti-commuting pair $(R_{-1}, T_3)$, we have that there is an anti-commuting pair of unitaries $(U, U_3)$ on a Hilbert space $\mathcal{K}$ and an isometry $V: \C^2 \to  \mathcal{K}$ such that 
\begin{equation}\label{eqn_323}
R_{-1}^mT_3^n=V^* U^mU_3^nV
\end{equation}
for all integers $m, n \geq 0$. Define
$(N_1, N_2, N_3)=(a_1U, c_2c_3^{-1}U_3U, U_3)$ on $\mathcal{K}$. It is evident that $\underline{N}$ is an anti-commuting triple of normal contractions on $\mathcal{K}$. Since $(U, U_3)$ is an anti-commuting pair of unitaries, we have that $N_iN_j^*=-N_j^*N_i$ for $1 \leq i < j \leq 3$. For non-negative integers $m, n, r$, we have by  (\ref{eqn_322}) \& (\ref{eqn_323}) that 
\begin{equation*}
\begin{split}
V^*\left(N_1^{m}N_2^{n}N_3^{r}\right)V
=a_1^m(c_2c_3^{-1})^nV^*\left(U^m(U_3U)^{n}U_3^r\right)V
=a_1^m(c_2c_3^{-1})^nR_{-1}^m(T_3R_{-1})^{n}T_3^r
=T_1^mT_2^nT_3^r
\end{split}
\end{equation*}
and hence, $\underline{N}$ dilates $\underline{T}$. The desired conclusion now follows from Theorem \ref{thm_dcq}.
\end{proof}

We now prove the following dilation result for an anti-commuting $k$-tuple of invertible contractions on $\C^2$. As mentioned earlier, it suffices to consider the case when $k=3$.

\begin{thm}\label{thm_antiII}
Let $\underline{T}=(T_1, T_2, T_3)$ be an anti-commuting triple of invertible contractions on $\C^2$. Then there exist an anti-commuting triple of unitaries $\underline{U}=(U_1, U_2, U_3)$ on a space $\mathcal{K} \supseteq \C^2$ and a non-zero scalar $\beta$ such that $\underline{N}=(\beta U_1, U_2, U_3)$ dilates $\underline{T}$. Moreover, if $\|T_1\| \leq \|T_2T_3\|$, then $(T_1,T_2,T_3)$ dilates to an anti-commuting triple of unitaries.
\end{thm}

\begin{proof}
Let $(T_1, T_2, T_3)$ be an anti-commuting triple of invertible contractions on $\C^2$. As mentioned in Proposition \ref{prop_anti}, $tr(T_j)=0$ for $j=1, 2, 3$. One can choose an orthonormal basis of $\C^2$ such that 
\[
T_1=\begin{pmatrix}
		a_1 & d_1 \\
		0 & -a_1\\
	\end{pmatrix} \quad \mbox{and} \quad 
	T_j=\begin{pmatrix}
		a_j & d_j\\
		c_j & -a_j\\
	\end{pmatrix} \quad (j=2,3).
\]
Note that $a_1 \ne 0$. If $d_1=0$, then the desired conclusion follows from Proposition \ref{prop_anti}. Let $d_1 \ne 0$. Some routine calculations show that
\[
T_1T_j+T_jT_1=\begin{pmatrix}
2a_1a_j+d_1c_j & 0 \\
0 & 2a_1a_j+d_1c_j
\end{pmatrix} \quad \text{and so,} \quad 2a_1a_j+d_1c_j=0 \quad (j=2,3).
\]
Hence, $c_j=0$ if and only if $a_j=0$ for $j=2,3$. It follows from the invertibility of $T_2$ and $T_3$ that $a_2, a_3, c_2, c_3$ are all non-zero. Therefore, we have that 
\begin{equation}\label{eqn_324}
c_2=\frac{-2a_1a_2}{d_1}, \quad c_3=\frac{-2a_1a_3}{d_1} \quad \text{and so,} \quad \frac{c_2}{c_3}=\frac{a_2}{a_3}=\lambda
\end{equation}
for some $\lambda \in \C \setminus \{0\}$. Thus, $T_2=\begin{pmatrix}
\lambda a_3 & d_2 \\
\lambda c_3 & -\lambda a_3 
\end{pmatrix}$ and it follows that
\[
 T_2T_3= \begin{pmatrix}
\lambda a_3^2 +d_2c_3 & \lambda a_3d_3-d_2a_3\\
0 & \lambda c_3d_3+\lambda a_3^2 \\
\end{pmatrix}
\ \
 \text{and} 
 \ \ T_3T_2=\begin{pmatrix}
\lambda a_3^2+ \lambda c_3d_3 & a_3d_2-\lambda a_3d_3 \\
0 & d_2 c_3 + \lambda a_3^2
\end{pmatrix}.
\]
Again, a few steps of calculations give
\begin{align} \label{eqn_325}
2T_2T_3=T_2T_3-T_3T_2
&= \begin{pmatrix}
c_3(d_2-\lambda d_3) & 2a_3(\lambda d_3-d_2) \\
0 & -c_3(d_2-\lambda d_3) \\
\end{pmatrix} \notag \\
&=c_3(d_2-\lambda d_3)\begin{pmatrix}
1 & -2a_3\slash c_3 \\
0 & -1
\end{pmatrix} \notag \\
&=a_1^{-1}c_3(d_2-\lambda d_3)\begin{pmatrix}
a_1 & -2a_1a_3\slash c_3 \\
0 & -a_1
\end{pmatrix} \notag \\
&=a_1^{-1}c_3(d_2-\lambda d_3)\begin{pmatrix}
a_1 & d_1 \\
0 & -a_1
\end{pmatrix} \qquad [\text{by} \ (\ref{eqn_324})] \notag \\
&=\alpha T_1, 
\end{align}
where $\alpha=a_1^{-1}c_3(d_2-\lambda d_3)$. Applying Theorem \ref{thm_KM} on the anti-commuting pair $(T_2, T_3)$, we have that there is an anti-commuting pair of unitaries $(U_2, U_3)$ on a Hilbert space $\mathcal{K}$ and an isometry $V: \C^2 \to \mathcal{K}$ such that 
\begin{equation}\label{eqn_326}
T_2^mT_3^n=V^* U_2^mU_3^nV
\end{equation}
for all integers $m, n \geq 0$. Set 
$
\underline{N}=(N_1, N_2, N_3)=(2\alpha^{-1}U_2U_3, U_2, U_3)=(\beta U_1, U_2, U_3)$, where $\beta =2 \alpha^{-1}$ and $U_1=U_2U_3$. Evidently, $\underline{N}$ is an anti-commuting triple of normal operators on $\mathcal{K}$. Since $U_2$ and $U_3$ are unitaries, we have that $N_iN_j^*=-N_j^*N_i$ for $1 \leq i < j \leq 3$. For integers $m, n, r \geq 0$, we have by  (\ref{eqn_325}) \& (\ref{eqn_326}) that 
\begin{equation*}
\begin{split}
V^*\left(N_1^{m}N_2^{n}N_3^{r}\right)V
=\beta^mV^*\left((U_2U_3)^{m}U_2^nU_3^r\right)V
=\beta^m(T_2T_3)^{m}T_2^nT_3^r
=T_1^mT_2^nT_3^r
\end{split}
\end{equation*}
and hence $\underline{N}$ dilates $\underline{T}$. In case, when $\|T_1\| \leq \|T_2T_3\|$, then it is clear from (\ref{eqn_325}) that $|\beta|\leq 1$ and so, $\underline{N}$ consists of invertible anti-commuting normal contractions. By Proposition \ref{prop_anti}, $\underline{N}$ dilates to an anti-commuting triple of unitaries and hence $(T_1,T_2,T_3)$ admits the desired dilation.
\end{proof}

Below we present an example of anti-commuting triple of invertible contractions $(T_1, T_2, T_3)$ on $\C^2$ satisfying $\|T_1\| \leq \|T_2T_3\|$. This is to show the existence of such an anti-commuting triple of $2 \times 2$ contractions.

\begin{eg}
For $\epsilon=0.1$, let 
$T_1=\epsilon^2\begin{pmatrix}
1 & 1 \\
0 & -1
\end{pmatrix}, T_2=\epsilon\begin{pmatrix}
1 & 0 \\
-2 & -1
\end{pmatrix}$ and $T_3=\epsilon\begin{pmatrix}
-1 & -1 \\
2 & 1
\end{pmatrix}$. It is not difficult to see that $(T_1, T_2, T_3)$ is an anti-commuting triple of invertible contractions on $\C^2$. Moreover, each $T_j$ is non-normal and  $T_2T_3=-T_1$ holds which further implies that $\|T_1\| \leq \|T_2T_3\|$.\qed
\end{eg}

Combining Theorems \ref{thm_ni} \& \ref{thm_antiII}, we arrive at the following theorem which is the main result of this Section.

\begin{thm}
Let $\underline{T}=(T_1, \dotsc, T_k)$ be an anti-commuting $k$-tuple of contractions on $\C^2$. Then there exist a $q$-commuting tuple $\underline{U}=(U_1, \dotsc, U_k)$ of unitaries on a space $\mathcal{K} \supseteq \C^2$ with $q \subseteq \{\pm 1\}$ and a non-zero scalar $\beta$ such that $\underline{N}=(\beta U_1, U_2, \dotsc, U_k)$ dilates $\underline{T}$. Moreover, the following two special cases arise.

\begin{enumerate}
\item If there exists a non-zero non-invertible contraction among $T_1, \dotsc, T_k$, then $(T_1, \dots, T_k)$ dilates to a $q$-commuting tuple of unitaries with $q \subseteq \{\pm 1 \}$.

\item If $T_1, \dotsc, T_k$ are all invertible, then $k\leq 3$ and thus $(T_1, \dots , T_k)$ reduces to at most to a triple say $(S_1,S_2,S_3)$. If $\|S_1\| \leq \|S_2S_3\|$, then $(S_1,S_2,S_3)$ dilates to an anti-commuting triple of unitaries.
\end{enumerate}
\end{thm}

\smallskip

\section{Classification of $q$-commuting tuples of $2 \times 2$ scalar matrix-contractions}\label{sec_02}
	
	\vspace {0.2cm}
	
	\noindent In this Section, we present a model for a $q$-commuting tuple of contractions acting on $\C^2$. We shall use the following notations for operators $T_1, \dotsc, T_k$ acting on a finite-dimensional space:
	\[
	\Lambda_1=\{ i: \ T_i \ \mbox{is diagonalizable}, \  1\leq i \leq k \}, \quad 
	\Lambda_2=\{i: \ T_i \ \mbox{is non-diagonalizable}, \ 1\leq i \leq k\}.
	\] 
We first discuss a few cases when a $q$-commuting tuple of contractions becomes a commuting tuple.
				
\begin{prop}\label{non-diag.}
		If $\underline{T}=(T_1, \dotsc, T_k)$ is a $q$-commuting tuple of non-diagonalizable contractions acting on $\mathbb{C}^2$, then $\underline{T}$ is a commuting tuple.
	\end{prop}
	
\begin{proof}
Fix $i \in \{1, \dotsc, k\}$. Choose an orthonormal basis $\beta$ of $\C^2$ with respect to which $T_i$ is upper triangular. So, with respect to the orthonormal basis $\beta$, we have
		\[
		T_i=\begin{pmatrix}
			\alpha_i & b_i\\
			0 & \alpha_i\\
		\end{pmatrix} \quad \text{and} \quad T_j=\begin{pmatrix}
			c_j & d_j\\
			e_j & f_j
		\end{pmatrix}  \quad ( j \ne i)
		\]
for some scalars $\alpha_i, b_i, c_j, d_j, e_j, f_j$. Evidently $b_i \ne 0$, because $T_i$ is non-diagonalizable. By the relation $T_iT_j=q_{ij}T_jT_i$, we have  
\begin{equation}\label{14.1}
			\alpha_ic_j+b_ie_j=q_{ij}\alpha_ic_j, \quad 
			\alpha_id_j+b_if_j=q_{ij}(c_jb_i+d_j\alpha_i),
		\end{equation}
	\begin{equation}\label{14.2}
			\alpha_ie_j=q_{ij}\alpha_ie_j, \quad 
			\alpha_if_j=q_{ij}(e_jb_i+f_j\alpha_i).
		\end{equation}
If $q_{ij}=1$, then $T_iT_j=T_jT_i$ and we are done. Suppose $q_{ij} \ne 1$. By  (\ref{14.2}), $\alpha_ie_j=0$ and so, if $\alpha_i \ne 0$, then $e_j=0$. Otherwise, let $\alpha_i = 0$. Again by (\ref{14.2}), $q_{ij}e_jb_i=0$ and so, $e_j=0$ since $q_{ij}$ and $b_i$ both are non-zero. In either case, $e_j=0$. Since $T_j$ is non-diagonalizable, we must have that $c_j=f_j$ and $d_j \ne 0$. Combining all these we have  
		\[
		T_i=\begin{pmatrix}
			\alpha_i & b_i\\
			0 & \alpha_i\\
		\end{pmatrix}, \quad T_j=\begin{pmatrix}
			c_j & d_j\\
			0 & c_j
		\end{pmatrix} \quad \mbox{and hence,} \quad
		T_iT_j=T_jT_i=\begin{pmatrix}
			\alpha_ic_j & b_ic_j+\alpha_id_j\\
			0 & \alpha_ic_j
		\end{pmatrix}.
		\]
The proof is now complete.
	\end{proof}

\begin{prop}\label{diag._scalar}
		If $\underline{T}=(T_1, \dotsc, T_k)$ is a $q$-commuting tuple of contractions on $\mathbb{C}^2$ such that every diagonalizable contraction in $\underline{T}$ is a scalar map $($i.e. a scalar times the identity operator$)$, then $\underline{T}$ is a commuting tuple.
	\end{prop}
	
\begin{proof}
Take any pair $(T_i, T_j)$ for $1 \leq i<j \leq k$. If $T_i$ or $T_j$ is in $\Lambda_1$, then one of them is a scalar map and so, $T_iT_j=T_jT_i$. If $T_i, T_j \in \Lambda_2$, then $T_i$ and $T_j$ commute by virtue of Proposition \ref{non-diag.}.
	\end{proof}

\begin{prop}\label{prop_204}
		Let $\underline{T}=(T_1, \dotsc, T_k)$ be a $q$-commuting tuple of contractions on $\mathbb{C}^2$ and let the spectrum $\sigma(T_i)=\{a_i, b_i\}$ for $1 \leq i \leq k$. If there exists $i \in \{1, \dotsc, k\}$ such that $T_i$ is diagonalizable with $|a_i|\ne |b_i|$, then $\underline{T}$ is a commuting tuple. 
	\end{prop}
	
	\begin{proof}
	Let $T_i$ be diagonalizable and let $|a_i|\ne |b_i|$ for some $i \in \{1, \dotsc, k\}$. Then there is a basis $\beta$ (not necessarily orthonormal) of $\mathbb{C}^2$ with respect to which 
		\[
		T_i=\begin{pmatrix}
			a_i & 0\\
			0 & b_i\\
		\end{pmatrix}  \quad \text{and} \quad T_j=\begin{pmatrix}
			c_j & d_j\\
			e_j & f_j\\
		\end{pmatrix} \quad (j \ne i)
		\]
for some scalars $c_j, d_j, e_j$ and $f_j$. By the relation $T_iT_j=q_{ij}T_jT_i$ and $|q_{ij}|=1$, we have that 
		\begin{equation}\label{14.3}
			a_ic_j=q_{ij}a_ic_j, \quad a_id_j=q_{ij}b_id_j,
		\end{equation}
		\begin{equation}\label{14.4}
			b_ie_j=q_{ij}a_ie_j, \quad b_if_j=q_{ij}b_if_j.
		\end{equation}
		By (\ref{14.3}), $(a_i-q_{ij}b_i)d_j=0$ and so, either $d_j=0$ or $a_i=q_{ij}b_i$. Now $a_i=q_{ij}b_i$ does not hold, because, $|a_i| \ne |b_i|$ and $|q_{ij}|=1.$ Consequently, $d_j=0.$ Similarly, it follows from (\ref{14.4}) that $e_j=0$. Thus, each $T_j$ becomes a diagonal matrix with respect to $\beta$ and the desired conclusion follows.
	\end{proof}

 Let $\underline{T}=(T_1, \dotsc, T_k)$ be a $q$-commuting tuple of non-zero contractions on $\mathbb{C}^2$ and let $\sigma(T_j)=\{a_j, b_j\}$ for $1 \leq j \leq k$. If $T_1, \dotsc, T_k$ are all non-diagonalizable, then Proposition \ref{non-diag.} implies that $\underline{T}$ is a commuting tuple. Suppose atleast one among $T_1, \dotsc, T_k$ is diagonalizable. If each diagonalizable entity in $\underline{T}$ is a scalar map, then $\underline{T}$ is a commuting tuple by  Proposition \ref{diag._scalar}. So, we assume that $\underline{T}$ consists of atleast one diagonalizable entity, say $T_i$, which is not a scalar map. If $|a_i| \ne |b_i|$, then $\underline{T}$ is a commuting tuple by Proposition \ref{prop_204}. 
 
\medskip  
 
 Now we discuss the case when the diagonalizable matrices among $T_1, \dotsc, T_k$ (that are not scalar maps) have eigenvalues of same modulus. Without loss of generality, we assume that $T_1, \dotsc, T_m$ are all diagonalizable and are not scalar maps with $|a_i|=|b_i|$ for $1 \leq i \leq m \leq k$. Take any $T_j$ for $m+1 \leq j \leq k$. If $T_j$ is a scalar map, then $a_j=b_j$. If $T_j$ is non-diagonalizable, then it must not have distinct eigenvalues, otherwise $T_j$ becomes diagonalizable. Thus, in this case as well, we have $a_j=b_j$. Therefore, we have that $|a_j|=|b_j|$ for all $j=1, \dotsc, k$ and thus, there exist $r_1, \dotsc, r_k$ in $\T$ such that $b_j=r_ja_j$ for $ 1\leq j \leq k$. Now we discuss the following three remaining cases:
	\begin{enumerate}
		\item[(a)] there is $j \in \{1, \dotsc, k\}$ such that $r_j \ne \pm 1$;
		\item[(b)] $r_j=1$ for $1 \leq j \leq k$; 
		\item[(c)] $r_j \in \{1,-1\}$ for  $1 \leq j \leq k$ and there is $ i \in \{ 1, \dotsc, k\}$ such that $r_i=-1$.
	\end{enumerate}

	\vspace{0.2cm}
	
		\noindent \textbf{Case (a).} Without loss of generality, let us assume that $r_1 \ne \pm 1$. Then $b_1=r_1a_1$ and one can choose a basis of $\C^2$ with respect to which 
		\[
T_1= \begin{pmatrix}
			a_1 & 0\\
			0 & r_1a_1\\
		\end{pmatrix} \quad \text{and}  \quad T_j= \begin{pmatrix}
			c_j & d_j\\
			e_j & f_j\\ 
		\end{pmatrix} \quad (1<j \leq k)
		\]
for some scalars $c_j, d_j, e_j$ and $f_j$. Since $(T_1, \dotsc, T_k)$ is $q$-commuting, we must have $T_1T_j=q_{1j}T_jT_1$ with $|q_{1j}|=1$. Then we have
	\begin{equation}\label{14.5}
		a_1c_j=q_{1j}a_1c_j, \quad a_1d_j=r_1q_{1j}a_1d_j,
	\end{equation}
	\begin{equation}\label{14.6}
		r_1a_1e_j=q_{1j}a_1e_j, \quad r_1a_1f_j=r_1q_{1j}a_1f_j.
	\end{equation}
	\begin{enumerate}

\item If $q_{1j}=1$, then it follows from (\ref{14.5}) \& (\ref{14.6})  that $(1-r_1)a_1d_j=0$ and $(1-r_1)a_1e_j=0$ respectively. Since $r_1 \ne 1$ and $a_1 \ne 0,$ we have that $d_j=0=e_j$ and so, $T_j=\begin{pmatrix}
			c_j & 0 \\
			0 & f_j
		\end{pmatrix}.$

\item If $q_{1j}=r_1,$ then $a_1d_j(1-r_1^2)=0$, $a_1(1-r_1)c_j=0$ and $a_1f_j(r_1-r_1^2)=0$ by (\ref{14.5}) \& (\ref{14.6}). Since $r_1 \ne \pm 1$ and $a_1 \ne 0,$ we have $c_j=d_j=f_j=0$. Thus, $T_j=\begin{pmatrix}
			0 & 0 \\
			e_j & 0
		\end{pmatrix}.$

\item If $q_{1j}=\overline{r_1}$, then $a_1c_j(1-\overline{r_1})=0, a_1(r_1-\overline{r_1})e_j=0$ and $a_1f_j(1-r_1)=0$ by  (\ref{14.5}) \& (\ref{14.6}). Since $r_1 \ne \pm 1$ and $a_1 \ne 0$, we have $c_j=e_j=f_j=0$ and so, $T_j=\begin{pmatrix}
			0 & d_j \\
			0 & 0
		\end{pmatrix}.$
	\end{enumerate}
If $q_{1j} \notin \{ 1, r_1, \overline{r_1}\}$, then we have by (\ref{14.5}) and (\ref{14.6}) that $c_j=d_j=e_j=f_j=0$.  Hence, $T_j=0$ which is a contradiction.	Therefore, $q_{1j} \in \{1, r_1, \overline{r_1}\}$. Let if possible, there are two indices $j \ne l$ such that $q_{1j}=r_1$ and $q_{1l}=\overline{r_1}$. From the above discussion and by $T_jT_l=q_{jl}T_lT_j$, we have that 
	\[T_j=\begin{pmatrix}
		0 & 0 \\
		e_j & 0
	\end{pmatrix}, \quad T_l=\begin{pmatrix}
		0 & d_l \\
		0 & 0
	\end{pmatrix}
	\quad \text{and so,} \quad 
	\begin{pmatrix}
		0 & 0 \\
		0 & e_jd_l
	\end{pmatrix}=\begin{pmatrix}
		q_{jl}d_le_j & 0 \\
		0 & 0
	\end{pmatrix}.
	\] 
	Thus, atleast one of $e_j$ or $d_l$ is zero which is a contradiction to the fact that all $T_i$'s are non-zero. Hence, either $q_{1j}\in \{1,r_1\}$ or $q_{1j} \in \{1, \overline{r_1}\}$ for $1 \leq j \leq k.$ 
		
	\vspace{0.2cm}
	
\noindent \textbf{Case (b).} Assume that $r_j=1$ for $1 \leq j \leq k$. In this case, every diagonalizable contraction among $T_1, \dotsc, T_k$ is a scalar map. It follows from Proposition \ref{diag._scalar} that 
	$(T_1, \dotsc, T_k)$ is a commuting tuple. 
		
	\vspace{0.2cm}

\noindent \textbf{Case (c).} Let us assume that $r_j \in \{1,-1\}$ for $1 \leq j \leq k$ and $r_j=-1$ for atleast one $j$. Without loss of generality, let us assume that $r_1=-1$, i.e. $a_1=-b_1.$ Choose a basis of $\C^2$ with respect to which we have 
\[
T_1= \begin{pmatrix}
	a_1 & 0\\
	0 & -a_1\\
\end{pmatrix} \quad \text{and}  \quad T_j= \begin{pmatrix}
	c_j & d_j\\
	e_j & f_j\\ 
\end{pmatrix} \quad (1<j \leq k)
\]
for some scalars $c_j, d_j, e_j$ and $f_j$. Using the relation $T_1T_j=q_{1j}T_jT_1$  and $a_1 \ne 0$, we have that 
\begin{equation}\label{14.7}
		c_j=q_{1j}c_j, \quad d_j=-q_{1j}d_j,
	\end{equation}
\begin{equation}\label{14.8}
		e_j=-q_{1j}e_j, \quad f_j=q_{1j}f_j.
	\end{equation}
If $q_{1j}\ne \pm 1$, then $c_j=d_j=e_j=f_j=0$ as $a_1 \ne 0.$ This leads to a contradiction as $T_j$ is non-zero and consequently, $q_{1j}\in \{1,-1\}$. The following are the remaining possibilities.

\begin{enumerate}
		\item Let $q_{1j}=1$. By (\ref{14.7}) \& (\ref{14.8}), $d_j=0=e_j$ and so, $T_j=\begin{pmatrix}
			c_j & 0\\
			0 & f_j
		\end{pmatrix}.$ Since $c_j$ and $f_j$ are eigenvalues of $T_j$ and $a_j=\pm b_j$,  we have that $\{c_j, f_j\}=\{a_j, b_j\}$ and $c_j=\pm f_j$.
		

\item Let $q_{1j}=-1$. We have by (\ref{14.7}) \& (\ref{14.8}) that $c_j=0=f_j$. Hence, $T_j=\begin{pmatrix}
			0 & d_j\\
			e_j & 0
		\end{pmatrix}$ for every $j \ne 1$. For any $l \ne j$ with $T_l=\begin{pmatrix}
			0 & d_l\\
			e_l & 0
		\end{pmatrix}$, the condition $T_jT_l=q_{jl}T_lT_j$ gives that 
		\[
		d_je_l=q_{jl}d_le_j \quad \text{and} \quad e_jd_l=q_{jl}e_ld_j.
		\]		
Since $T_j, T_l$ are non-zero, it follows that $d_j=0$ if and only if $d_l=0$. Also, $e_j=0$ if and only if $e_l=0$. Hence, if both $d_j, e_j$ are non-zero, then $d_l, e_l$ are non-zero as well.  
	\end{enumerate} 

\smallskip 

Thus, we arrive at the following model for a tuple $(T_1, \dotsc, T_k)$ of $q$-commuting contractions acting on $\C^2$. To explain it clearly, we write $\eta_\alpha:=\{i : 1 \leq i \leq k, \ T_1T_i=\alpha T_iT_1 \}$ for $\alpha \in \C$.

\begin{thm}\label{structure thm}
		Let $\underline{T}=(T_1, \dotsc, T_k)$ be a $q$-commuting tuple of non-zero contractions acting on $\mathbb{C}^2$. Then either $\underline{T}$ is commuting or $($upto a permutation of indices$)$ there is a basis of $\C^2$ $($not necessarily orthonormal$)$ with respect to which we have the following.
		\begin{enumerate}
			\item[\textbf{Type-$I$}:] If $T_1=\begin{pmatrix}
				a & 0 \\
				0 & ra\\
			\end{pmatrix}$ with $a \ne 0, r \in \T\setminus \{\pm 1\}$, then $q_{1j}\in \{1, r\}$ and we have
			\[
			T_j=\left\{
			\begin{array}{ll}
				\begin{pmatrix}
					c_j & 0 \\
					0 & f_j\\
				\end{pmatrix} \ \  (|c_j|=|f_j| \ne 0)& \mbox{if} \ j \in \eta_1, \\ \\
				\begin{pmatrix}
					0 & 0 \\
					e_j & 0\\
				\end{pmatrix} \ \  (e_j \ne 0) &  \mbox{if} \ j \in \eta_r, 
			\end{array} 
			\right. \quad \quad (2 \leq j \leq  k).
			\]	

\item[\textbf{Type-$II$}:] If $T_1=\begin{pmatrix}
				a & 0 \\
				0 & ra\\
			\end{pmatrix}$ with $a \ne 0, r \in \T \setminus \{\pm 1\}$, then $q_{1j}  \in \{1, \overline{r}\}$ and we have
			\[
			T_j=\left\{
			\begin{array}{ll}
				\begin{pmatrix}
					c_j & 0 \\
					0 & f_j\\
				\end{pmatrix} \ \  (|c_j|=|f_j| \ne 0)& \mbox{if} \ j \in \eta_1, \\ \\
				\begin{pmatrix}
					0 & d_j \\
					0 & 0\\
				\end{pmatrix} \ \  (d_j \ne 0) &  \mbox{if} \ j \in \eta_{\overline{r}},
			\end{array} 
			\right. \quad \quad (2 \leq j \leq  k).
			\]	

\item[\textbf{Type-$III$}:] If $T_1=\begin{pmatrix}
				a & 0 \\
				0 & -a\\
			\end{pmatrix}$ with $a \ne 0$ then $q_{1j} \in \{1, -1\}$ and we have
			\[
			T_j=\left\{
			\begin{array}{ll}
				\begin{pmatrix}
					c_j & 0 \\
					0 & \pm c_j\\
				\end{pmatrix} \ \  (c_j \ne 0)& \mbox{if} \ j \in \eta_1, \\ \\
				\begin{pmatrix}
					0 & d_j \\
					e_j & 0\\
				\end{pmatrix} \ \  (d_j \ne 0 \ \text{or} \  e_j \ne 0) &  \mbox{if} \ j \in \eta_{-1}, 
			\end{array} 
			\right. \quad \quad  (2 \leq j \leq  k).
			\]	
		\end{enumerate}
			\end{thm}

\smallskip 

\begin{rem} 
Note that Type-$I$ and Type-$II$ correspond to $q_{1j}\in \{1, r\}$ and $q_{1j} \in \{1, \overline{r}\}$ in Case (a) respectively, and that Type-$III$ corresponds to Case (c). The following diagram summarizes Theorem \ref{structure thm}.
	
	
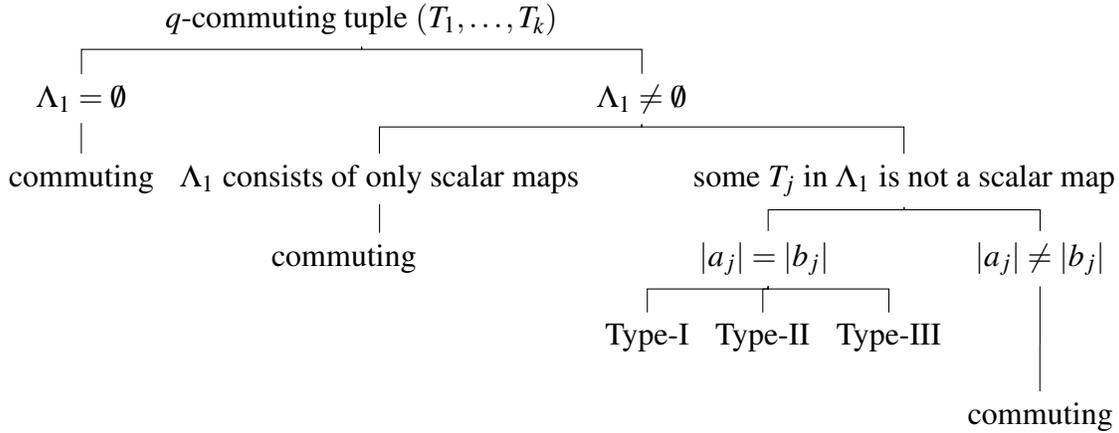
\vspace{0.2cm}
	\begin{center}
		\begin{tikzpicture}
			\tikzset{edge from parent/.style={draw,edge from parent path={(\tikzparentnode.south)-- +(0,-1pt)-| (\tikzchildnode)}}}
			\Tree [.\text{$q$-commuting tuple $(T_1, \dotsc, T_k)$}
			[.\text{$\Lambda_1 = \emptyset$}
			[.\text{commuting}
			] ]
			[.\text{$\Lambda_1 \ne \emptyset$}
			[.\text{$\Lambda_1$ consists of only scalar maps}
			[.\text{commuting \qquad}  ] ]
			[. \text{some $T_j$ in $\Lambda_1$ is not a scalar map}
			[.\text{$|a_j| = |b_j|$ }  
			[.\text{Type-I} ]
			[.\text{Type-II} ] 
			[.\text{Type-III} ] ] 
			[.\text{$|a_j|\ne |b_j|$}  [\text{commuting} ] ] ] ] ] ]
		\end{tikzpicture}
	\end{center}
\end{rem}

The next result gives one more sufficient condition for a $q$-commuting tuple of contractions on $\C^2$ to become a commuting tuple. It is a simple application of Theorem \ref{structure thm}.

\begin{cor}\label{cor_q_comm}
Let $\underline{T}=(T_1, \dotsc, T_k)$ be a $q$-commuting tuple of contractions on $\C^2$. If $T_j$ has no distinct eigenvalues for $1 \leq j \leq k$, then $\underline{T}$ is commuting. 
\end{cor}

\begin{proof}
Without loss of generality, we assume that $\underline{T}$ consists of non-zero contractions. We have by Theorem \ref{structure thm} that $\underline{T}$ is either commuting or it is similar to a $q$-commuting tuple of contractions $\underline{T}'=(T_1' \dotsc, T_k')$  on $\C^2$ that is of Type-$I$, Type-$II$ or Type-$III$. Since similar matrices have same eigenvalues, we have by the hypothesis that $T_j'$ has no distinct eigenvalues for $1 \leq j \leq k$. 

\medskip 

\noindent \textbf{Case (a).} Suppose that $\underline{T}'$ is of Type-$I$. Then one can choose $m \in \{1, \dotsc, k\}$ such that 
\[
T_j'=\begin{pmatrix}
c_j' & 0\\
0 & f_j'\\
\end{pmatrix} \quad \text{and} \quad  T_i'=\begin{pmatrix}
0 & 0 \\
e_i' & 0\\
\end{pmatrix}
\]
with $|c_j'|=|f_j'|$ for $1 \leq j \leq m$ and $m+1 \leq i \leq k$. Since the eigenvalues of $T_j'$ are equal, we have that $c_j'=f_j'$ for $1 \leq j \leq m$. Consequently, $\underline{T}'$  and hence, $\underline{T}$ is a commuting tuple. 

\medskip

\noindent \textbf{Case (b).} Suppose that $\underline{T}'$ is of Type-$II$. Then $\underline{T}'^*=(T_1'^*, \dotsc, T_k'^*)$ is of Type-$I$ and thus, by Case (a), $(T_1^*, \dotsc, T_k^*)$ is a commuting tuple. Hence, $\underline{T}$ is a commuting tuple.

\medskip

\noindent \textbf{Case (c).} If $\underline{T}'$ is of Type-$III$, then we have by Theorem \ref{structure thm} that there is $j \in \{1, \dotsc, k\}$ with $\sigma(T_j')=\{a_j, -a_j\}$ for some non-zero scalar $a_j$. This is a contradiction to the fact that the eigenvalues of $T_j'$ are not distinct. The proof is now complete.
\end{proof}

\medskip

\section{Dilation of some classes of $2 \times 2$ $q$-commuting contractive matrices} \label{sec_03}

\medskip

\noindent As we have mentioned earlier that a commuting tuple of contractions acting on $\C^2$ always dilates to a commuting tuple of unitaries, see \cite{Holbrook} by Holbrook. In this Section, we find a $\widetilde{q}$-unitary dilation for a $q$-commuting tuple of contractions on $\C^2$ that is unitarily equivalent to one of Type-$I$, Type-$II$ or Type-$III$ as in Theorem \ref{structure thm}, where $\widetilde q \subseteq q \cup \{1\}$. Then for a general $q$-commuting tuple of contractions on $\C^2$, we obtain upto similarity a $\widetilde{q}$-commuting dilation that consists of a scalar multiple of unitaries. Let us begin with the first step.

\begin{thm}\label{thm_dil_III}
Let $\underline{T}=(T_1, \dotsc, T_k)$ be a $q$-commuting tuple of contractions on $\C^2$. If $\underline{T}$ is unitarily equivalent to a $q$-commuting tuple of contractions on $\C^2$ that is of Type-$I$ or Type-$II$ or Type-$III$ as in Theorem \ref{structure thm}, then $\underline{T}$ admits a $\widetilde{q}$-unitary dilation with $\widetilde{q} \subseteq q \cup \{1\}$. 
\end{thm}		

	\begin{proof} 
Suppose $\underline{T}$ is unitarily equivalent to a $q$-commuting tuple of contractions on $\C^2$, say $\underline{T}'=(T_1', \dotsc, T_k')$, that is of Type-$I$ or Type-$II$ or Type-$III$. In general, if a tuple of commuting contractions $(T_1, \dotsc, T_k)$ on a Hilbert space $\HS$ dilates to a commuting tuple of unitaries $(U_1, \dotsc, U_k)$ on a space $\mathcal{K} \supseteq \HS$, then with respect to a possible different embedding of $\HS$ into $\mathcal{K}$, $(U_1, \dotsc, U_k)$ becomes a dilation of a unitarily equivalent tuple $(U_0T_1U_0^{-1}, \dotsc, U_0T_kU_0^{-1})$, for any unitary $U_0$ on $\mathcal{H}$. Evidently, the same holds in the $q$-commuting setting. Therefore, we assume without loss of generality that $\underline{T}$ falls into one of Type-$I$, Type-$II$, Type-$III$. 

\medskip 

\noindent \textbf{Case I.} Assume that $\underline{T}$ is of Type-$I$. We have by Theorem \ref{structure thm} that 	
\[
T_1=\begin{pmatrix}
	a & 0 \\
	0 & ra\\
\end{pmatrix}
 \quad \text{and} \quad 		T_j=\left\{
		\begin{array}{ll}
			\begin{pmatrix}
				c_j & 0 \\
				0 & f_j\\
			\end{pmatrix} \ \  (|c_j|=|f_j| \ne 0)& \mbox{if} \ j \in \eta_1, \\ \\
			\begin{pmatrix}
				0 & 0 \\
				e_j & 0\\
			\end{pmatrix} \ \  (e_j \ne 0) &  \mbox{if} \ j \in \eta_{r},
		\end{array} 
		\right. 
		\]	
where  $a \ne 0, r \in \T \setminus \{\pm 1\}$  and $q_{1j} \in \{1, r \}$ for $2 \leq j \leq  k$. Some routine calculations yield 
\begin{equation}\label{eqn_209}
	T_iT_j=\left\{
\begin{array}{ll}
	T_jT_i & \mbox{if} \ i, j \in \eta_1, \\ 
	\alpha_i T_jT_i & \mbox{if} \ i \in \eta_1, j \in \eta_{r}, \\
	T_jT_i & \mbox{if} \ i, j \in \eta_{r}, 
\end{array} 
\right.  \quad \text{where} \quad \alpha_i=\left\{
\begin{array}{ll}
	r & \mbox{if} \ i=1,\\ 
	f_i \slash c_i  & \mbox{if} \ i\ne 1. 
\end{array} 
\right.
\end{equation}
Note that $|\alpha_i|=1$ for every $i \in \eta_1$. Let $i, j \in \eta_1$. It is evident that $T_iT_j \ne 0$. Then $T_iT_j=q_{ij}T_jT_i=T_jT_i$ and so, $q_{ij}=1$. Let $i \in \eta_1, j \in \eta_r$. Again $T_iT_j \ne 0$. Thus, $T_iT_j=q_{ij}T_jT_i=\alpha_i T_jT_i$ and so,  $q_{ij}=\alpha_i$. For $i, j \in \eta_r$, we have $T_iT_j=0=T_jT_i$. In this case, we replace the $q_{ij}$ by $1$. Let $\widetilde{q}=\{\widetilde{q}_{ij}: 1 \leq i<j \leq k\}$ be such that $\widetilde{q}_{ij}=q_{ij}$ if $i$ or $j$ is in $\eta_1$ and $\widetilde{q}_{ij}=1$ if $i, j \in \eta_r$. Evidently, $\underline{T}$ is a $\widetilde{q}$-commuting tuple and $\widetilde{q} \subseteq q \cup \{1\}$. Let $m \in \eta_r$ be such that $|e_m|=\max_{j \in \eta_r}|e_j|$. Define $w_j=e_j\slash e_m$ for $j \in \eta_r$. As per notations, $1 \in \eta_1$ and $c_1=a$. Again, by some simple calculations, we have that 
\begin{equation}\label{eqn_210}
T_j=w_jT_m \quad \text{and} \quad T_i=c_iR_{\alpha_i} \quad (j \in \eta_r, \ i \in \eta_1),
\end{equation}
where we define 
\[
R_\alpha=\begin{pmatrix}
1 & 0 \\
0 & \alpha
\end{pmatrix} \quad (\alpha \in \C).
\]
It is not difficult to see that 
\begin{equation}\label{eqn_211}
\begin{split}
R_{\alpha_i}T_m=\alpha_iT_mR_{\alpha_i}, 
\quad
 R_{\alpha_i}^*T_m=\overline{\alpha}_iT_mR_{\alpha_i}^*, 
 \quad 
 R_{\alpha_i}R_{\alpha_\ell}=R_{\alpha_\ell}R_{\alpha_i} 
\quad \text{and} \quad  
R_{\alpha_i}^*R_{\alpha_\ell}=R_{\alpha_\ell}R_{\alpha_i}^*
\end{split}
\end{equation}
for every $i, \ell \in \eta_1$. Therefore, $\{T_m, R_{\alpha_i} : i \in \eta_1\}$ is a doubly $q'$-commuting family of contractions on $\C^2$, where the family $q'$ of unimodular scalars is determined by (\ref{eqn_211}). It follows from part-(2) of Theorem \ref{thm_dcq} that there is a $q'$-commuting family of unitaries $\{U_m, \widetilde{R_{\alpha_i}} : i \in \eta_1 \}$ on a space $\mathcal{K}$ and an isometry $V:\C^2 \to \mathcal{K}$ such that
\begin{equation}\label{eqn_212}
T_m^{s_1}\prod_{i \in \eta_1}R_{\alpha_i}^{s_i}=V^* U_m^{s_1}\prod_{i \in \eta_1}\widetilde{R_{\alpha_i}}^{s_i}V
\end{equation}
for all non-negative integers $s_1, s_i$ with $i \in \eta_1$ and that
\begin{equation}\label{eqn_213}
\begin{split}
\widetilde{R_{\alpha_i}}U_m=\alpha_iU_m\widetilde{R_{\alpha_i}},
 \quad 
 \widetilde{R_{\alpha_i}}^*U_m=\overline{\alpha}_iU_m\widetilde{R_{\alpha_i}}^*, 
 \quad  
 \widetilde{R_{\alpha_i}}\widetilde{R_{\alpha_\ell}}=\widetilde{R_{\alpha_\ell}}\widetilde{R_{\alpha_i}}
  \quad 
 \text{and} \quad 
 \widetilde{R_{\alpha_i}}^*\widetilde{R_{\alpha_\ell}}=\widetilde{R_{\alpha_\ell}}\widetilde{R_{\alpha_i}}^*. \\
\end{split}
\end{equation}
We define an operator tuple $\underline{N}=(N_1, \dotsc, N_k)$ in the following way:
\begin{equation*}
N_i=\left\{
\begin{array}{ll}
	c_i\widetilde{R_{\alpha_i}} & \mbox{if} \ i\in \eta_1,\\ 
	w_iU_m   & \mbox{if} \ i \in \eta_r. 
\end{array} 
\right.
\end{equation*}
It is evident that each $N_i$ is a normal contraction on $\mathcal{K}$. Moreover, we have by (\ref{eqn_213}) that

\begin{equation}\label{eqn_214}
	N_iN_j=\left\{
\begin{array}{ll}
	N_jN_i & \mbox{if} \ i, j \in \eta_1, \\ 
	\alpha_i N_jN_i & \mbox{if} \ i \in \eta_1, j \in \eta_{r}, \\
	N_jN_i & \mbox{if} \ i, j \in \eta_{r} 
\end{array} 
\right. \quad \text{and} \quad 
N_iN_j^*=\left\{
\begin{array}{ll}
	N_j^*N_i & \mbox{if} \ i, j \in \eta_1, \\ 
	\overline{\alpha}_i N_j^*N_i & \mbox{if} \ i \in \eta_1, j \in \eta_{r}, \\
	N_j^*N_i & \mbox{if} \ i, j \in \eta_{r}. 
\end{array} 
\right.
\end{equation}
We show that $T_1^{m_1}\dotsc T_k^{m_k}=V^*N_1^{m_1}\dotsc N_k^{m_k}V$ for all non-negative integers $m_1, \dotsc, m_k$. It is clear from (\ref{eqn_214}) that $\underline{N}$ is a doubly $\widetilde{q}$-commuting tuple. Therefore, it suffices to  show that $T_{\sigma(1)}^{m_1}\dotsc T_{\sigma(k)}^{m_k}=V^*N_{\sigma(1)}^{m_1}\dotsc N_{\sigma(k)}^{m_k}V$ for some permutation $\sigma$ on $\{1, \dotsc, k\}$ and for all non-negative integers $m_1, \dotsc, m_k$. Without loss of generality, we can assume that $\{1, \dotsc, i\} \in \eta_1$ and $\{i+1, \dotsc, k\} \in \eta_r$. For non-negative integers $m_1, \dotsc, m_k$, we have that 
\begin{equation*}
\begin{split}
V^*N_1^{m_1} \dotsc N_k^{m_k}V
=V^*\prod_{1 \leq j \leq i}N_j^{m_j}\prod_{i+1 \leq j \leq k}N_j^{m_j}V
&=V^*\prod_{1 \leq j \leq i}c_j^{m_j}\widetilde{R_{\alpha_j}}^{m_j}\prod_{i+1 \leq j \leq k}w_j^{m_j}U_m^{m_j}V\\
&=\prod_{1 \leq j \leq i}c_j^{m_j}R_{\alpha_j}^{m_j}\prod_{i+1 \leq j \leq k}w_j^{m_j}T_m^{m_j} \qquad \qquad [\text{by} \ (\ref{eqn_212})]\\ 
&=\prod_{1 \leq j \leq i}T_j^{m_j}\prod_{i+1 \leq j \leq k}T_j^{m_j} \qquad \qquad \qquad \quad [\text{by} \ (\ref{eqn_210})]\\
&=T_1^{m_1} \dotsc T_k^{m_k}.
\end{split}
\end{equation*}
Therefore, $\underline{N}$ is a doubly $\widetilde{q}$-commuting tuple of contractions that dilates $\underline{T}$. We have the desired conclusion from Theorem \ref{thm_dcq}.

\medskip 

\noindent \textbf{Case II.}
Assume that $\underline{T}$ is of Type-$II$. It follows from Theorem \ref{structure thm} that 
\[
T_1=\begin{pmatrix}
	a & 0 \\
	0 & ra\\
\end{pmatrix}
 \quad \text{and} \quad 		T_j=\left\{
		\begin{array}{ll}
			\begin{pmatrix}
				c_j & 0 \\
				0 & f_j\\
			\end{pmatrix} \ \  (|c_j|=|f_j| \ne 0)& \mbox{if} \ j \in \eta_1, \\ \\
			\begin{pmatrix}
				0 & d_j \\
				0 & 0\\
			\end{pmatrix} \ \  (e_j \ne 0) &  \mbox{if} \ j \in \eta_{\overline{r}},
		\end{array} 
		\right. 
		\]	
where  $a \ne 0, r \in \T \setminus \{\pm 1\}$  and $q_{1j} \in \{1, \overline{r} \}$ for $2 \leq j \leq  k$. Note that for $|q|=1$, a pair $(A, B)$ of operators satisfies $AB=qBA$ if and only if $A^*B^*=qB^*A^*$. Thus, $\underline{T}^*=(T_1^*, \dotsc, T_k^*)$ is also a $q$-commuting tuple and is of Type-$I$. The desired conclusion follows from Case I.
\medskip 

\noindent  We mention that the proofs presented in Cases I and II also work when $r \in \T \setminus \{1\}$, i.e. one can choose $r=-1$ in both the above cases without any change in their respective proofs.
\medskip

\noindent \textbf{Case III.} Let $\underline{T}$ be of Type-$III$. We have by Theorem \ref{structure thm} that

 		\[
		T_1=\begin{pmatrix}
	a & 0 \\
	0 & -a\\
\end{pmatrix}
 \quad \text{and} \quad 		T_j=\left\{
		\begin{array}{ll}
			\begin{pmatrix}
				c_j & 0 \\
				0 & \alpha_j c_j\\
			\end{pmatrix} \ \  (c_j \ne 0, \alpha_j=\pm 1)& \mbox{if} \ j \in \eta_1, \\ \\
			\begin{pmatrix}
				0 & d_j \\
				e_j & 0\\
			\end{pmatrix} \ \  (d_j \ne 0 \ \text{or} \ e_j \ne 0) &  \mbox{if} \ j \in \eta_{-1},
		\end{array} 
		\right. 
		\]	
where  $a \ne 0$  and $q_{1j} \in \{\pm 1\}$ for $2 \leq j \leq  k$. If there exists $i \in \{1, \dotsc, k\}$ such that $d_i=0$, then it follows from the proof of Theorem \ref{structure thm} that $d_j=0$ for all $j \in \eta_{-1}$. The desired conclusion now follows from Case I. Suppose that there exists $i \in \{1, \dotsc, k\}$ such that $e_i=0$, then it follows from the proof of Theorem \ref{structure thm} that $e_j=0$ for all $j \in \eta_{-1}$. We have the desired conclusion from Case II.

\medskip 

We assume that $d_j, e_j$ are non-zero scalars for $j \in \eta_{-1}$. It is not difficult to see that $T_iT_j \ne 0$ for $1 \leq i<j \leq k$. Hence, we have unique unimodular scalars $q_{ij}$ such that $T_iT_j=q_{ij}T_jT_i$ holds for all $i, j$. Let $m \in\eta_{-1}$ be such that $ |d_m|=\underset{j\in \eta_{-1}}{\max}|d_j|$. Take $j\in \eta_{-1}$. We have by $T_mT_j=q_{mj}T_jT_m $ that $	e_md_j=q_{mj}e_jd_m=q_{mj}^2d_je_m$ and so, $q_{mj}=\pm 1$. Therefore, $d_je_m=q_{mj}d_me_j$ so that $q_{mj}\dfrac{e_j}{e_m}=\dfrac{d_j}{d_m}$ for all $j\in \eta_{-1}$. Our choice of $m \in\eta_{-1}$ implies that $\left|\dfrac{e_j}{e_m}\right|=\left|\dfrac{d_j}{d_m}\right|\leq 1$. Consequently, $|e_m|=\underset{j\in \eta_{-1}}{\max}|e_j|$. Let $w_j=\dfrac{d_j}{d_m}$ for $j \in \eta_{-1}$. Hence, $w_j\in \overline{\mathbb{D}}$ and $e_j=q_{mj}w_je_m$ for all $j\in \eta_{-1}$. Again, some routine calculations yield
\begin{equation}\label{eqn10.12}
	T_iT_j=\left\{
\begin{array}{ll}
	T_jT_i & \mbox{if} \ i, j \in \eta_1, \\ 
	\alpha_i T_jT_i & \mbox{if} \ i \in \eta_1, j \in \eta_{-1}, \\
	q_{mi}q_{mj} T_jT_i & \mbox{if} \ i, j \in \eta_{-1} 
\end{array} 
\right. 
\end{equation}
and
\begin{equation}\label{eqn109}
		T_j=w_jR_{q_{mj}}T_m \quad \text{and} \quad T_i=c_iR_{\alpha_i} \quad (j \in \eta_{-1}, i \in \eta_1).
	\end{equation}
It is easy to see that  $R_{-1}T_m=-T_mR_{-1}$ and $R_{-1}^*T_m=-T_mR_{-1}^*$. Thus, $(R_{-1}, T_m)$ is an anti-commuting pair of contractions. We have by Theorem \ref{thm_dcq} that there is an anti-commuting pair of unitaries $(\widetilde{R_{-1}}, U_m)$ acting on a space $\mathcal{K}$ and an isometry $V: \C^2 \to \mathcal{K}$ such that 
\begin{equation}\label{eqn_219}
V^*\widetilde{R_{-1}}^{t_1}U_m^{t_2}V=(R_{-1})^{t_1}T_m^{t_2}  
\end{equation}
for all non-negative integers $t_1, t_2$. For $i\in \eta_1$ and $j\in \eta_{-1}$, let us define 
\[
\widetilde{R_{\alpha_i}}=\left\{
\begin{array}{ll}
	I_\mathcal{K} & \mbox{if} \ \alpha_i=1, \\ 
	\widetilde{R_{-1}} & \mbox{if} \ \alpha_i=-1
\end{array} 
\right. \quad \text{and} \quad \widetilde{R_{q_{mj}}}=\left\{
\begin{array}{ll}
	I_\mathcal{K} & \mbox{if} \ q_{mj}=1, \\ 
	\widetilde{R_{-1}} & \mbox{if} \ q_{mj}=-1. 
\end{array} 
\right.
\]
From this definition, it is clear that 
		\begin{equation}\label{eqn1012}
\widetilde{R_{\alpha_i}}\widetilde{R_{q_{mj}}}=\widetilde{R_{q_{mj}}}\widetilde{R_{\alpha_i}}, \quad 
		\widetilde{R_{\alpha_i}}U_m=\alpha_iU_m\widetilde{R_{\alpha_i}} \quad \text{ and } \quad \widetilde{R_{q_{mj}}}U_m=q_{mj}U_m\widetilde{R_{q_{mj}}} 
		\end{equation}
for $i \in \eta_1$ and $j \in \eta_{-1}$. Let us define an operator tuple $\underline{N}=(N_1, \dotsc, N_k)$ as follows:		
\[
		N_i=\left\{
		\begin{array}{ll}
				c_i\widetilde{R_{\alpha_i}} & \mbox{if} \ i \in \eta_1, \\ 
				
			w_i\widetilde{R_{q_{mi}}}U_m & \mbox{if} \ i \in \eta_{-1}. 
		\end{array} 
		\right.
		\]
It is evident that each $N_i$ is a normal contraction on $\mathcal{K}$. Moreover, we have by (\ref{eqn1012}) that
\begin{equation}\label{eqn_221}
		N_iN_j=\left\{
		\begin{array}{ll}
			N_jN_i & \mbox{if} \ i, j \in \eta_1, \\ 
			\alpha_i N_jN_i & \mbox{if} \ i \in \eta_1, j \in \eta_{-1}, \\
			q_{mi}q_{mj} N_jN_i & \mbox{if} \ i, j \in \eta_{-1} 
		\end{array} 
		\right. \quad \text{and} \quad 
N_iN_j^*=\left\{
\begin{array}{ll}
	N_j^*N_i & \mbox{if} \ i, j \in \eta_1, \\ 
	\overline{\alpha}_i N_j^*N_i & \mbox{if} \ i \in \eta_1, j \in \eta_{-1}, \\
	\overline{q}_{mi}\overline{q}_{mj}N_j^*N_i & \mbox{if} \ i, j \in \eta_{-1}. 
\end{array} 
\right.
	\end{equation}
For every non-negative integers $m_1, \dotsc, m_k$, we show that 
			\begin{equation}\label{eqn_222}
		V^*N_1^{m_1}\dots N_k^{m_k}V=T_1^{m_1}\dots T_k^{m_k}.
	\end{equation}
It is clear from (\ref{eqn10.12}) \& (\ref{eqn_221}) that the contractions in the tuples $\underline{N}, \underline{T}$ satisfy the same intertwining relations and $\underline{N}$ is a doubly $q$-commuting tuple. Thus, (\ref{eqn_222}) is equivalent to showing that 
\begin{equation}\label{eqn_223}
V^*N_{\sigma(1)}^{m_1}\dots N_{\sigma(k)}^{m_k}V=T_{\sigma(1)}^{m_1}\dots T_{\sigma(k)}^{m_k}
\end{equation}
for some permutation $\sigma$ on $\{1, \dotsc, k\}$ and all non-negative integers $m_1, \dotsc, m_k$. For operators $A$ and $B$ with $AB=-BA$ and $n \in \N \cup \{0\}$, it is easy to verify that
\begin{equation}\label{eqn_theta}
(AB)^n=(-1)^{\theta(n)} A^nB^n \quad \text{where} \quad \theta(n)=\frac{1}{2}n(n-1).
\end{equation}
Choose a permutation $\sigma$ such that 

\begin{enumerate}
	\item[(i)] $\sigma(1), \dotsc, \sigma(i) \in \eta_1$ and $\alpha_{\sigma(1)}=\dotsc = \alpha_{\sigma(i)} =1$;
	\item[(ii)] $\sigma(i+1), \dotsc, \sigma(j) \in \eta_1$ and $\alpha_{\sigma(i+1)}=\dotsc = \alpha_{\sigma(j)}=-1$;
	\item[(iii)] $\sigma(j+1), \dotsc, \sigma(l) \in \eta_{-1}$ and $q_{m\sigma(j+1)}=\dotsc = q_{m\sigma(l)} =1$;
	\item[(iv)] $\sigma(l+1), \dotsc, \sigma(k) \in \eta_{-1}$ and $q_{m\sigma(l+1)}=\dotsc = q_{m\sigma(k)} = -1$, 
\end{enumerate}
where $ 1 \leq i \leq j \leq l \leq k$. Take $m_1, \dotsc, m_k$ in $\N \cup \{0\}$. For the ease of computations, we define
\[
c_\sigma=c_{\sigma(1)}^{m_1}\dots c_{\sigma(j)}^{m_j}, \quad  w_\sigma=w_{\sigma(j+1)}^{m_{j+1}}\dots w_{\sigma(k)}^{m_k} \quad \text{and}  \quad \theta(m)=\theta(m_{l+1})+\dotsc+\theta(m_k). 
\]
Note that
\begin{equation}\label{eqn1201}
\begin{split}
		& \quad \ N_{\sigma(1)}^{m_1}\dots N_{\sigma(k)}^{m_k}\\
		&= \bigg[N_{\sigma(1)}^{m_1}\dots N_{\sigma(i)}^{m_i}\bigg]
		\bigg[N_{\sigma(i+1)}^{m_{i+1}}\dots N_{\sigma(j)}^{m_j}\bigg]
	    \bigg[N_{\sigma(j+1)}^{m_{j+1}}\dots N_{\sigma(l)}^{m_l}\bigg]
		\bigg[N_{\sigma(l+1)}^{m_{l+1}}\dots N_{\sigma(k)}^{m_k}\bigg]\\
		&= \bigg[c_{\sigma(1)}^{m_1}\dots c_{\sigma(i)}^{m_i}\bigg]
		\bigg[c_{\sigma(i+1)}^{m_{i+1}}\dots c_{\sigma(j)}^{m_j}\widetilde{R_{-1}}^{m_{i+1}+\dotsc + m_j}\bigg]
		\bigg[w_{\sigma(j+1)}^{m_{j+1}}\dots w_{\sigma(l)}^{m_l}U_m^{m_{j+1}+\dotsc + m_l}\bigg] \\
		& \qquad 
		\bigg[w_{\sigma(l+1)}^{m_{l+1}}\dots w_{\sigma(k)}^{m_k}(\widetilde{R_{-1}}U_m)^{m_{l+1}}\dotsc (\widetilde{R_{-1}}U_m)^{m_{k}}\bigg]\\	
		&=(-1)^{\theta(m)} c_\sigma w_\sigma
		\bigg[\widetilde{R_{-1}}^{m_{i+1}+\dotsc + m_j}
		U_m^{m_{j+1}+\dotsc + m_l}\bigg]
		\bigg[(\widetilde{R_{-1}}^{m_{l+1}}U_m^{m_{l+1}})\dotsc (\widetilde{R_{-1}}^{m_{k}}U_m^{m_{k}})\bigg].\\
\end{split}
\end{equation}
One can use induction to show that 
\begin{equation}\label{eqn1202}
(\widetilde{R_{-1}}^{m_{l+1}}U_m^{m_{l+1}})\dotsc (\widetilde{R_{-1}}^{m_{k}}U_m^{m_{k}})=(-1)^p\widetilde{R_{-1}}^{m_{l+1}+\dotsc + m_k}U_m^{m_{l+1}+\dotsc + m_k},
\end{equation}
where $p=m_{l+1}m_{l+2}+(m_{l+1}+m_{l+2})m_{l+3}+\dotsc +(m_{l+1}+\dotsc +m_{k-1})m_k$. Then
\begin{equation}\label{eqn1203}
	\begin{split}
	 \ N_{\sigma(1)}^{m_1}\dots N_{\sigma(k)}^{m_k}
		&=(-1)^{\theta(m)+p} c_\sigma w_\sigma
		\bigg[\widetilde{R_{-1}}^{m_{i+1}+\dotsc + m_j}
		U_m^{m_{j+1}+\dotsc + m_l}\bigg]
		\bigg[\widetilde{R_{-1}}^{m_{l+1}+\dotsc + m_k}U_m^{m_{l+1}+\dotsc + m_k}\bigg]\\	
			&=(-1)^{\theta(m)+p} c_\sigma w_\sigma
		\widetilde{R_{-1}}^{m_{i+1}+\dotsc + m_j}
		\bigg(U_m^{m_{j+1}+\dotsc + m_l}
		\widetilde{R_{-1}}^{m_{l+1}+\dotsc + m_k}\bigg)U_m^{m_{l+1}+\dotsc + m_k}\\	
		&= (-1)^{\theta(m)+p+t} c_\sigma w_\sigma
		\bigg(\widetilde{R_{-1}}^{m_{i+1}+\dotsc + m_j+m_{l+1}+ \dotsc +m_k}\bigg)
		\bigg(U_m^{m_{j+1}+\dotsc + m_k}\bigg),
	\end{split}
\end{equation}
where $t=(m_{j+1}+\dotsc + m_l)(m_{l+1}+\dotsc + m_k)$. Therefore, we have by (\ref{eqn_219}) \& (\ref{eqn1203}) that
\[
V^* N_{\sigma(1)}^{m_1}\dots N_{\sigma(k)}^{m_k}V=(-1)^{\theta(m)+p+t} c_\sigma w_\sigma
\bigg(R_{-1}^{m_{i+1}+\dotsc+m_j+m_{l+1}+\dotsc + m_k}\bigg)
\bigg(T_m^{m_{j+1}+\dotsc + m_k}\bigg).
\]
Following the same computations as in (\ref{eqn1201}), (\ref{eqn1202}), (\ref{eqn1203}) for $T_{\sigma(1)}, \dotsc, T_{\sigma(k)}$, we have that 
\[
T_{\sigma(1)}^{m_1}\dots T_{\sigma(k)}^{m_k}= (-1)^{\theta(m)+p+t}c_\sigma w_\sigma
\bigg(R_{-1}^{m_{i+1}+\dotsc +m_j+m_{l+1}+\dotsc+ m_k}\bigg)
\bigg(T_m^{m_{j+1}+\dotsc + m_k}\bigg)
\]
and so, (\ref{eqn_223}) holds. Putting everything together, we have proved that $\underline{N}$ is a doubly $q$-commuting tuple of contractions which dilates $\underline{T}$. The desired conclusion follows from Theorem \ref{thm_dcq} and the proof is now complete. 
	\end{proof}

For a general $q$-commuting tuple of contractions on $\C^2$, we have the following dilation theorem which can be considered as a corollary of Theorem \ref{thm_dil_III}.

\begin{thm}\label{thm_dilation_2}
		Let $\underline{T}=(T_1,\dots ,T_k)$ be a $q$-commuting tuple of contractions acting on $\mathbb{C}^2$. Then $\underline{T}$ is either a commuting tuple or there exist an invertible operator $P$ on $\C^2$ and a $\widetilde{q}$-commuting tuple of unitaries $(U_1, \dotsc, U_k)$ on a Hilbert space $\mathcal{K}$ containing $\HS$ such that $(\beta U_1, \dotsc, \beta U_k)$ dilates $(P^{-1}T_1P, \dotsc, P^{-1}T_kP)$, where $\beta=\|P^{-1}\| \|P\|$ and $\widetilde{q} \subseteq q \cup \{1\}$.
	\end{thm} 

\begin{proof}
Suppose that $\underline{T}=(T_1,\dots ,T_k)$ is a $q$-commuting tuple of contractions acting on $\mathbb{C}^2$. By Theorem \ref{structure thm}, $\underline{T}$ is either commuting or $\underline{T}$ is similar to a $q$-commuting tuple of contractions on $\C^2$ that is one of Type-$I$, Type-$II$, Type-$III$. In the latter case, there is an invertible operator $P$ on $\C^2$ such that $\underline{A}:=(A_1, \dotsc, A_k)=(P^{-1}T_1P, \dotsc, P^{-1}T_kP)$ is one of Type-$I$, Type-$II$ and Type-$III$. Note that the tuple $\underline{A}$ may or may not consist of contractions. Let $\beta=\|P^{-1}\|\|P\|$. Then $(\beta^{-1}A_1, \dotsc, \beta^{-1}A_2)$ is a $q$-commuting tuple of contractions that has same type as of the original tuple $\underline{T}$. We have by Theorem \ref{thm_dil_III} that there is a Hilbert space $\mathcal{K}$, an isometry $V: \C^2 \to \mathcal{K}$ and a $\widetilde{q}$-commuting tuple of unitaries $(U_1, \dotsc, U_k)$ on $\mathcal{K}$ with $\widetilde{q} \subseteq q \cup \{1\}$ such that
\[
(\beta^{-1}A_1)^{m_1}\dotsc (\beta^{-1}A_k)^{m_k}=V^*U_1^{m_1}\dotsc U_k^{m_k}V \ \ \text{and so,} \ \
 P^{-1}T_1^{m_1}\dotsc T_k^{m_k}P=V^*(\beta U_1)^{m_1}\dotsc (\beta U_k)^{m_k}V
\] 
for all non-negative integers $m_1, \dotsc, m_k$. The proof is now complete.
\end{proof}

 The next step to the results of this article is to determine if a tuple of $q$-commuting contractions on $\C^2$ admits a $q$-unitary dilation. In the present paper, we classified such $q$-commuting tuples in three types and found a $\widetilde q$-unitary dilation for each type or more generally for $q$-commuting tuples that are unitarily equivalent to each type, where $\widetilde q  \subseteq q \cup \{1\}$. However, neither we could prove that a $q$-commuting $k$-tuple of $2 \times 2$ matrix-contractions dilates to a $q$-commuting $k$-tuple of unitaries, nor we could produce any counter example to establish that such a $q$-commuting tuple on $\C^2$ does not dilate to any $q$-commuting tuple of unitaries. Hence, the question if Holbrook's result \cite{Holbrook} can be generalized in the $q$-commuting setting remains unresolved at this point.
 
 \vspace{0.3cm}
 
 \noindent \textbf{Acknowledgements.} The first named author is supported by the Seed Grant of IIT Bombay, the CDPA and the `Core Research Grant' with Award No. CRG/2023/005223 of Science and Engineering Research Board (SERB), India. The second named author is supported by the Ph.D Fellowship of Council of Scientific and Industrial Research (CSIR), India. The third named author is supported by the Prime Minister's Research Fellowship (PMRF ID 1300140), Government of India.

\end{document}